\documentclass[12pt,a4paper]{amsart}
\usepackage[abbrev]{amsrefs}
\usepackage{amssymb}
\usepackage{enumerate}
\usepackage{color}
\usepackage{mathtools}

\theoremstyle{plain}
  \newtheorem{theorem}{Theorem}[section]
  \newtheorem{lemma}[theorem]{Lemma}
  \newtheorem{corollary}[theorem]{Corollary}
  \newtheorem{proposition}[theorem]{Proposition}

\theoremstyle{definition}
  \newtheorem{definition}[theorem]{Definition}

\theoremstyle{remark}
  \newtheorem{remark}[theorem]{Remark}
  \newtheorem*{ack}{Acknowledgment}

\numberwithin{equation}{section}

\makeatletter
\newcommand\RedeclareMathOperator{%
  \@ifstar{\def\rmo@s{m}\rmo@redeclare}{\def\rmo@s{o}\rmo@redeclare}%
}
\newcommand\rmo@redeclare[2]{%
  \begingroup \escapechar\m@ne\xdef\@gtempa{{\string#1}}\endgroup
  \expandafter\@ifundefined\@gtempa
     {\@latex@error{\noexpand#1undefined}\@ehc}%
     \relax
  \expandafter\rmo@declmathop\rmo@s{#1}{#2}}
\newcommand\rmo@declmathop[3]{%
  \DeclareRobustCommand{#2}{\qopname\newmcodes@#1{#3}}%
}
\@onlypreamble\RedeclareMathOperator
\makeatother

\newcommand{\N}{\ensuremath{\mathbb{N}}}
\newcommand{\R}{\ensuremath{\mathbb{R}}}

\begin{document}
\title[{Convergence of combinatorial Ricci flows to degenerate circle patterns}]
{Convergence of combinatorial Ricci flows\\ to degenerate circle patterns}

\author{Asuka Takatsu}
\address{\scriptsize Department of Mathematical Sciences, Tokyo Metropolitan University, Tokyo 192--0397, JAPAN}% ({\sf asuka@tmu.ac.jp})}
\email{asuka@tmu.ac.jp}
\address{\scriptsize Mathematical  Analysis Team, RIKEN Center for Advanced Intelligence Project (AIP), Tokyo 103--0027, JAPAN}

\date{\today}


\thanks{This work was supported by JSPS KAKENHI Grant Numbers 15K17536, 16KT0132.}

%%%%%%%%%%%%%%%%%%%%%%%%%%%%%%%%%%%%%%
\begin{abstract}
We investigate the combinatorial Ricci flow on a surface of nonpositive Euler characteristic
when the necessary and sufficient condition for the convergence of the combinatorial Ricci flow is not valid.
This observation addresses one of questions raised by B.~Chow and F.~Luo.
\end{abstract}

%%%%%%%%%%%%%%%%%%%%%%%%%%%%%%%%%%%%%%
\maketitle
%%%%%%%%%%%%%%%%%%%%%%%%%%%%%%%%%%%%%%
\section{Introduction}
%%%%%%%%%%%%%%%%%%%%%%%%%%%%%%%%%%%%%%
On a connected, oriented, closed surface $S$, 
although its triangulation  $T=(V,E,F)$ is a \emph{topological} structure,
a metric of constant curvature zero on a weighted triangulation $(T,\Theta)$ uniquely determines a pair of a \emph{geometric} structure of $S$
and a configuration of circles on $S$ realizing the data $(T,\Theta)$ (see~\cite{Th}).
%%%%
Here $V,E,F$ respectively denote the set of all vertices, edges and triangles of $T$, 
and we assume that the \emph{triangulation} $T$ lifts to a triangulation in the universal cover of $S$.
A \emph{weighted triangulation} $(T,\Theta)$ is a triangulation equipped with a \emph{weight} $\Theta:E \to [0,\pi/2]$.
We call $r=(r_v)_{v\in V} \in \R_{>0}^{V}$ a \emph{metric} on $(T,\Theta)$ 
and define the \emph{length} $\ell(e;r)$ of an edge $e$ with endpoints $v, u$ with respect to $r$ by
\[
\ell(e;r)
:=
\left(\begin{tabular}{c}
the distance between the centers of two circles on the universal cover of $S$\\
with radii $r_v, r_u$ intersecting at angle $\Theta(e)$
\end{tabular}
\right),
\]
which determines the angle at each vertex in triangles of $(T,\Theta)$.
The \emph{cone angle} $a_v(r)$ at the vertex $v$ with respect to $r$ is the sum of each angle at $v$ in all triangles of $(T,\Theta)$ having~$v$ as one of the vertices.
%%%
We call $K_v(r):=2\pi-a_v(r)$ the \emph{curvature} at the vertex $v$ with respect to $r$.
A metric is called a \emph{circle pattern metric} provided its curvature is identically zero at each vertex.

A circle pattern metric on $(T, \Theta)$, which connects the topology with the geometry of $S$, does not always exist.
However, a criterion for $(T,\Theta)$ to have a circle pattern metric is known, 
which depends on whether the Euler characteristic $\chi(S)$ of $S$ is positive or nonpositive.
The case of $\chi(S) \leq 0$ was first obtained by Thurston \cite{Th}.
%%%
Chow--Luo \cite{CL} gave another proof by using the combinatorial Ricci flow,  which is a family of metrics $\{r(t)\}_{t\in[0,T)}$ that satisfies 
\begin{equation*}
\frac{d}{dt}r_v(t)=-K_v(r(t))\sigma_S(r_v(t)),
\quad
\sigma_S(r_v):=\begin{dcases}
r_v, & \text{if\ $\chi(S)=0$},\\
\sinh r_v, & \text{if\ $\chi(S)<0$}.
\end{dcases}
\end{equation*}
%%%%%%%%%%%%%%%%%%%%%%%%%%%%%%%%%%%%%%%%%%%%%%%%%%%%%%%%%%%%%%%%%%%%%%
%%%%%%%%%%%%%%%%%%%%%%%%%%%%%%%%%%%%%%%%%%%%%%%%%%%%%%%%%%%%%%%%%%%%%%
\begin{theorem} {\rm (\cite{Th}*{Theorem 13.7.1}, \cite{CL}*{Theorems 1.1, 1.2})} \label{CLT}
Let $(T,\Theta)$ be a weighted triangulation of a surface $S$ of nonpositive Euler characteristic.
%%%
The following conditions ${\rm (I),(II)}$ and ${\rm(III)}$ are equivalent to each other.$:$
\begin{enumerate}
%%%
\item[\rm{(I)}]
There exists a unique circle pattern metric, up to a scalar multiple if $\chi(S)=0$.
%%%
\item[\rm{(II)}]
It holds for any nonempty proper subset $U$ of $V$ that
\[
\phi(U):=-\sum_{f\in \mathrm{Lk}(U)} ( \pi-\Theta( e^f_{v(f)}) ) +2\pi \chi(\tau_{U}) <0,
\]
where $\tau_{U}$ is the CW-subcomplex of $T$, consisting of all cells whose vertices are contained in~$U$,
and  $f\in\mathrm{Lk}(U)$ is an element in $F$ 
such that one vertex $v(f)$ in $f$ belongs to $U$ and 
neither of the endpoints of the edge $e^f_{v(f)}$ in $f$ opposite to $v(f)$ belongs to~$U$.
%%%
\item [\rm{(III)}]
Given any metric $r$ on $(T,\Theta)$,
the combinatorial Ricci flow with initial data $r$ exists for all time and converges on $\R^{V}_{>0}$ at infinity.
\end{enumerate}
\end{theorem}
%%%%%%%%%%%%%%%%%%%%%%%%%%%%%%%%%%%
We remark that, instead of (II),
the following condition~(II$'$) is used in~\cite{Th}*{Theorem 13.7.1} and \cite{CL}*{Theorem 1.2}.:
\begin{enumerate}
\item[\rm{(II$'$)}] 
$(T,\Theta)$ satisfies the following two conditions.:
\begin{itemize}
\item[\rm{(i)}]
If $e_1,e_2,e_3 \in E$ form a null-homotopic loop and if $\sum_{i=1}^3 \Theta(e_i)\geq \pi$,
then these three edges form the boundary of a triangle of $T$.
%%%%
\item[\rm{(ii)}]
If $e_1,e_2,e_3,e_4 \in E$ form a null-homotopic loop and if $\sum_{i=1}^4 \Theta(e_i)= 2\pi$,
then these four edges form the boundary of the union of two adjacent triangles of $T$.
\end{itemize}
\end{enumerate}
%%%%%%%%%%%%%%%%%%%%%%%%%%%%%%%%%%%
The condition~(II) implies (II$'$), and they are equivalent to each other if $\chi(S)<0$.
%%%
Although it is stated in \cite{CL}*{Proposition 1.3} that (II) follows from (II$'$) even if $\chi(S)=0$,
there is a counterexample (see Remark \ref{counter}).
Note that 
Thurston~\cite{Th}*{\S 13.7} treated a cell-division of not only triangles but also quadrangles by adding following condition:
\begin{enumerate}
\item[(IV)]
$\Theta(e)=\pi/2$ if $e$ is an edge of a quadrilateral. 
\end{enumerate}
Such a weighted quadrangle can be eliminated by subdivision into two weighted triangles by a diagonal of weight 0.

Chow--Luo~\cite{CL}*{\S 7} raised some questions about the combinatorial Ricci flow,
one of which is to investigate the combinatorial Ricci flow  when (II) is not valid.
We address it when the maximum of $\phi(U)$ over all nonempty proper subsets $U$ of $V$ is zero.
%%%%%%%%%%%%%%%%%%%%%%%%%%%%%%%%%%%%%%%%%%%%%
\begin{theorem}\label{main}
Let $(T,\Theta)$ be a weighted triangulation of a surface $S$ of nonpositive Euler characteristic 
such that $\phi(U) \leq 0$ holds for any $U\subset V$ and 
\[
Z_T:=\{z\in V \ |\ \text{there exists a proper subset $Z$ of $V$ such that  $z\in Z$  and  $\phi(Z)=0 $}\}
\]
is nonempty.
Then for any metric $r$ on $(T,\Theta)$,
the combinatorial Ricci flow $\{r(t)\}_{t\geq0}$ with initial data $r$ does not converge on $\R^{V}_{>0}$ at infinity, however
\[
\lim_{t \to \infty} K_v(r(t))=0
\]
holds for any vertex $v$. 

On the one hand,
for $\chi(S)=0$,  $\{r(t)\}_{t\geq0}$ does not converge on $\R^{V}_{\geq 0}$ at infinity.
However if we fix an arbitrary $v \in V\setminus Z_T$,
then the limit
\[
\rho_u:=\lim_{t \to \infty} \frac{r_u(t)}{r_{v}(t)}
\]
exists for any  $u\in V$,
%%%
where $Z_T=\{z\in V\ |\ \rho_z=0\}$ holds and 
$(\rho_u)_{u\in V\setminus Z_T}$ is a unique circle patten metric with normalization $\rho_{v}=1$
on a certain weighted triangulation with vertices~$V\setminus Z_T$.

On the other hand, for $\chi(S)<0$, $\{r(t)\}_{t\geq0}$ converges on $\R^{V}_{\geq 0}$ at infinity,
where we have
$Z_T=\{z\in V\ |\ \lim_{t \to \infty}r_z(t)=0\}$ holds and 
the limit of $(r_v(t))_{v\in V\setminus Z_T}$ at infinity is a unique circle patten metric on a certain weighted triangulation with vertices~$V\setminus Z_T$.

\end{theorem}

It should be mentioned that weighted triangulations on a torus not satisfying (II$'$-i) are previously studied by Yamada~\cite{Y}, 
where a part of Theorem \ref{main} was stated and partially proved for $|V| \leq 3$ by analyzing individual cases.
We present a completely different proof for the general cases, based on an infinitesimal description of degenerate circle patten metrics.

The organization of this paper is as follows. 
Section 2 is devoted to recalling the relevant facts and terminology from  weighted triangulations, circle pattern metrics and the combinatorial Ricci flow.
In Section 3,  we give the infinitesimal description of degenerate circle patten metrics and prove Theorem \ref{main}.

\begin{ack}
The author would like to thank Professor Hiroshi Matsuzoe for bringing this topic to her attention and stimulating discussions.
She also expresses her thanks to Mr.~Masahiro Yamada for seminars and discussions.
\end{ack}

%%%%%%%%%%%%%%%%%%%%%%%%%%%%%%%%%%%%%%%%%%%%%%%%%%%%%%
%%%%%%%%%%%%%%%%%%%%%%%%%%%%%%%%%%%%%%%%%%%%%%%%%%%%%%
\section{Preliminaries}
Throughout this paper, let $S$ be a connected, oriented, closed surface of nonpositive Euler characteristic.
%%%
Most of the results in this section hold true for a general surface if modified appropriately.
See \cite{Th}, \cite{CL}, \cite{V} \cite{MR} and references therein for details.
%%%%%%%%%%%%%%%%%%%%%%%%%%%%%%%%%%%%%%%%%%%%%%%%%%%%%%
%%%%%%%%%%%%%%%%%%%%%%%%%%%%%%%%%%%%%%%%%%%%%%%%%%%%%%
\subsection{Weighted Triangulation}
A \emph{triangulation} $T=(V,E,F)$ of $S$ is always assumed to lift to a triangulation in the universal cover of $S$,
where $V,E,F$ respectively denote the set of all vertices (0-cells), edges (1-cells) and triangles (2-cells) of $T$.
In particular, there is no null-homotopic loop formed by at most two edges.
%%%%%%
Although a 2-cell $f\in F$ is an open set in~$S$, if its boundary $\partial f$ contains a vertex $v\in V$ and an edge $e\in E$, 
then we write $v\in f$ and $e\subset f$ instead of $v\in\partial f$ and $e\subset \partial f$, respectively.
Similarly for $U\subset V$, we write $f\cap U$ and $e \cap U $ instead of  $\partial f \cap U$ and $\partial e \cap U$, respectively.
Let us fix a \emph{weighted triangulation} $(T,\Theta)$ of $S$, 
where $\Theta:E\to[0,\pi/2]$ is a \emph{weight}.
%%%
%%%%%%%%%%%%%%%

%%%%%%%%%%%%%%%%%%%%%%%%%%%%%%%
\begin{definition}
A pair of vertices is said to be \emph{adjacent} if they are joined by an edge.
\end{definition}
%%%%%%%%%%%%%%%%%%%%%%
\begin{definition}
For $U \subset V$,
define the \emph{link of $U$} by
\[
\mathrm{Lk}(U):=\cup_{v\in U}\{f\in F |\ \text{$v\in f$ and there exists $e\in E$ such that $e\subset f$ and $e\cap U=\emptyset$}\}.
\]
%%%%
For $f\in \mathrm{Lk}(U)$, 
we denote by $v(f)$  a unique element in $U$ such that $v(f)\in f$.
%%%
\end{definition}
%%%%%%%%%%%%%%%%%%%%%%%%%%%%%%%
\begin{definition}
For $U \subset V$, let $\tau_{U}$ be the CW-subcomplex of $T$, consisting of all cells whose vertices are contained in $U$.
%%%
We say that $U$ is \emph{connected} if $\tau_{U}$ is connected.
Set
\[
\mathcal{C}(U):=\{W \subset U |\ \text{$\tau_W$ is a connected component of $\tau_{U}$} \}.
\]
\end{definition}
%%%%%%%%%%%%%%%%%%%%%%%%%%%%%%%%
If $\tau_{U}$ is contractible, then $\{e^f_{v(f)}\}_{f\in\mathrm{Lk}(U)}$ form a null-homotopic loop,
where $e^f_{v(f)}$ is the edge in $f$ opposite to $v(f)$.
However the converse is not true in general. 
For $j=1,2,3$, define
\[
F_{U}^j:=\{f\in F\ |\ \text{$j$ vertices of $f$ are in $U$ }\},
\]
then $F_{U}^1=\mathrm{Lk}(U)$ and the number of edges (resp.\ triangles) of $\tau_{U}$ is $(|F_{U}^2|+3|F_{U}^3|)/2$ (resp.\ $|F_{U}^3|$), 
consequently
\[
\chi(\tau_{U})=|U|-(|F_{U}^2|+|F_{U}^3|)/2. %\frac12(|F_{U}^2|+|F_{U}^3|).
\]
%%%%%%%%%%%%%%%%%%

Let us consider the function defined for $U\subset V$ by
\[
\phi(U):=-\sum_{f\in \mathrm{Lk} (U)}(\pi-\Theta(e^f_{v(f)}))+2\pi\chi(\tau_{U}),
\]
%%%
where the right-hand side is additive with respect to connected components of $\tau_{U}$.

%%%%
\begin{proposition}{\rm (\cite{Th}, \cite{CL}*{Proofs of Corollaries 4.2,4.3}, \cite{MR})}\label{CW}.
%%%
Let $U$ be a nonempty proper connected subset of $V$.
If $\phi(U)>0$, then $\tau_{U}$ is contractible and 
we have
\[
\sum_{f\in \mathrm{Lk}(U)}\Theta(e^f_{v(f)})>\pi, \qquad
|\mathrm{Lk}(U)|=3.
\]
If $\phi(U)=0$, then $\tau_{U}$  is contractible and 
we have
\[
\sum_{f\in\mathrm{Lk}(U)}\Theta(e^f_{v(f)})=(|\mathrm{Lk}(U)|-2)\pi,
\quad
|\mathrm{Lk}(U)|\in\{3,4\}.
\]
\end{proposition}
%%%%
\begin{remark}\label{counter}
%%%%
If $e_1,e_2,e_3 \in E$ form a null-homotopic loop but not the boundary of a triangle of~$T$,
then there exits $U\subset V$ such that 
$\tau_{U}$  is contractible and $\{e_i\}_{i=1}^3=\{e^f_{v(f)}\}_{f\in \mathrm{Lk}(U)}$,
%%%
in which case $\sum_{i=1}^3\Theta(e_i) \geq \pi$ is equivalent to $\phi(U)\geq0$.
%%%%
Thus the contrapositive of the implication from (II) to (II$'$-i) and similarly to (II$'$-ii) holds.
%%%%

Conversely, if there exists a nonempty proper connected subset $U$ of $V$ such that $\phi(U)\geq 0$, 
then $\{e^f_{v(f)}\}_{f\in \mathrm{Lk}(U)}$ form a null-homotopic loop
and $\sum_{f\in \mathrm{Lk}(U)} \Theta(e^f_{v(f)}) \geq (|\mathrm{Lk}(U)|-2)\pi$ by Proposition~\ref{CW}.
The null-homotopic loop divides $S$ into open sets  $S_1,S_2$ with $U\subset S_1$ (possibly $S_2=\emptyset$ if $\chi(S)=0$).
We find $|S_1 \cap F|\geq 3$ and, for $k:=\min\{|S_1 \cap F|,|S_2 \cap F|\}$,
the null-homotopic loop forms the boundary of the union of at least $k$-adjacent triangles of $T$.
Set $F_1$ as $S_1$ itself if $|\mathrm{Lk}(U)|=3$, and as two triangles obtained by dividing $S_1$ by a diagonal if $|\mathrm{Lk}(U)|=4$.
Then $F':=F_1 \cup \{f\in F\ |\ f\cap S_2 \neq \emptyset \}$ determines a new triangulation~$(V',E', F')$ of $S$.
It follows from the property $3|F'|=2|E'|$ that
\[
2-2\chi(S) \leq 2(|V'|-\chi(S)) =2(|E'|-|F'|)=|F'|=|F_1|+|S_2 \cap F| \leq 2+|S_2 \cap F|,
\]
implying  $|S_2 \cap F| \geq -2\chi(S)$.
Hence if $\chi(S)<0$ then $k\geq 3$, which contradicts (II$'$).
%%%
However for $\chi(S)=0$, namely $S=S^1 \times S^1$, 
if we choose $V=\{v_j=(e^{\sqrt{-1}\pi j/2},e^{-\sqrt{-1}\pi j/2} )\}_{j=1}^2,  U=\{v_1\}$ 
and if $\{e^f_{v_1}\}_{f\in \mathrm{Lk}(\{v_1\})}$ consists of 
the three curves 
\[
\{(e^{\sqrt{-1}\theta}, e^{\sqrt{-1}\pi})\}_{\theta \in (-\pi,\pi)}, \quad 
\{(e^{\sqrt{-1}\pi}, e^{\sqrt{-1}\theta})\}_{\theta \in (-\pi,\pi)},\quad
\{(e^{\sqrt{-1}\theta}, e^{\sqrt{-1}\theta})\}_{\theta \in (-\pi,\pi)}
\]
of weight $\pi/3$, 
then $|S_2\cap F|=1$, implying that $\{e^f_{v_1}\}_{f\in \mathrm{Lk}(\{v_1\})}$ form the boundary of a triangle, 
and this weighted triangulation has no circle pattern metric.
Thus (II$'$) does not imply (II) and (II$'$) is not a sufficient condition to have a circle pattern metric when $\chi(S)=0$.

However, for a weighted triangulation of a torus on which any triangle has three distinct vertices,
(II$'$) implies (II)  as well as the case of $\chi(S)<0$.
\end{remark}
%%%%%%%%%%%
\begin{definition}
%%%
For $k=3,4$, define
\begin{align*}
\mathcal{Z}_k:=\{ Z\subset V\  |\ \text{$Z$ is connected and $\phi(Z)=0$ with $|\mathrm{Lk}(Z)|=k$}\},\quad 
\mathcal{Z}:=\mathcal{Z}_3 \cup \mathcal{Z}_4.
\end{align*}
For $Z\in\mathcal{Z}$, 
we denote by $f_Z$ a unique 2-cell containing $Z$ with boundary $\{e^f_{v(f)}\}_{f\in\mathrm{Lk}(Z)}$. 
\end{definition}
%%%%%%%
%
Note that $f_Z \notin F$ and $Z_T=\{z\in Z \ |\ Z \in \mathcal{Z}\}$.
For $Z \in \mathcal{Z}$, 
there exists $U\in \mathcal{C}(Z_T)$ such that $Z\subset U$,
but $Z\notin \mathcal{C}(Z_T)$ in general.
%%%
However any $Z\in \mathcal{C}(Z_T)$ satisfies $Z \in \mathcal{Z}$.
%%%%%%%%%%%%

\begin{definition}
Given $Z\in\mathcal{Z}_4$, 
a \emph{diagonal edge} $e$ in $f_Z$ is a 1-cell contained in $f_Z$,
for which there exist $e_1,e_2 \in \{e^f_{v(f)}\}_{f\in\mathrm{Lk}(Z)}$ such that 
$e,e_1,e_2$ form a null-homotopic loop.
\end{definition}

\begin{remark}
For $Z\in\mathcal{Z}_4$, there are two diagonal edges in $f_Z$, neither of which belongs to $E$.
A diagonal edge $e$ in $f_Z$ divides $f_Z$ into two triangles $f_Z^1,f_Z^2$,
%%%%
and if the weight of $e$ is 0,
then, at each endpoint of $e$, the angle in $f_Z$ is same as the sum of the angles in $f_Z^1, f_Z^2$.
Thus $f_Z$ can be divided into two weighted triangles as well as a weighted quadrangle satisfying~(IV).
%%%
\end{remark}

%%%%%%%%%%%%%%%%%%%%
\subsection{Circle Pattern Metrics}
%%%
For a weighted triangulation $(T,\Theta)$ of $S$,
we call $r \in \R^{V}_{>0}$ a \emph{metric} on $(T,\Theta)$.
We define the \emph{length} $\ell(e;r)$ of an edge $e$ with endpoints $v,u$ with respect to $r$ by 
\[
\ell(e;r):=
\begin{dcases}
\displaystyle\sqrt{r_v^2+r_u^2+2r_vr_u\cos \Theta(e)}, & \text{if $\chi(S)=0$},\\
\mathrm{arccosh}\left(\cosh r_v \cosh r_u +\sinh r_v\sinh r_u\cos \Theta(e)\right), & \text{if $\chi(S)<0$},
\end{dcases}
\] 
which is the distance between the centers of two circles on the universal cover of $S$
with radii $r_v, r_u$ intersecting at angle $\Theta(e)$.
Needless to say, if $\chi(S)=0$, then $S$ is a torus and its universal cover is the Euclidean plane $\mathbb{E}^2$.
%%%%
On the other hand, the universal cover of a surface of negative Euler characteristic is the hyperbolic plane $\mathbb{H}^2$.
%

%%%%%%%%%%%%
Let us first recall the existence of a unique circle pattern in the universal cover of $S$ 
and properties of the triangle formed by the centers of the three circles.
\begin{proposition}{\rm(\cite{Th}*{Lemmas 13.7.2, 13.7.3}, \cite{CL}*{Lemma 2.3},\cite{V})} \label{angle_sym} 
For any three nonobtuse angles $\Theta_1,\Theta_2,\Theta_3 \in [0,\pi/2]$
and any three radii $r_1, r_2,r_3>0$,
there is a configuration of three circles $C_1,C_2,C_3$ in both $\mathbb{E}^2$ and $\mathbb{H}^2$, 
unique up to isometry, having radii $r_i$ and meeting in angles~$\Theta_i$.
%%%%
Moreover, for the triangle formed by the centers of the three circles $C_1,C_2,C_3$,
if we regard the angle $\theta_i$ at the center of $C_i$ in the triangle
as a function of $r_1,r_2,r_3$,
then we have
\[
\frac{\partial \theta_i}{\partial r_i}<0<\frac{\partial \theta_i}{\partial r_j}, \quad
 \frac{\partial \theta_i}{\partial r_j} \sigma_S(r_j)=\frac{\partial \theta_j}{\partial r_i} \sigma_S(r_i)
\]
for any $1\leq i, j\leq 3$ with $i\neq j$.
\end{proposition}
%
%%%%%%%%%%%%%%%%%%%%%%%%%%%%%%%
%
Although three edges of any triangle of $T$ differ from each other, vertices of a triangle may happen to coincide.
When one vertex $v$ of $f\in F$ is different from the remaining two vertices of $f$,
we denote by $e^f_v$ the edge opposite to $v$ in $f$, and by $\theta^f_v(r)$ the angle at $v$ in $f$ with respect to $r$.
%%%%%
In this case, we have
\begin{equation*}
\cos\theta_{v}^f(r)=
\begin{dcases}
\dfrac{\ell(e;r)^2+\ell(e';r)^2-\ell(e_v^f; r)^2}{2\ell(e;r)\ell(e';r)}, &\text{if $\chi(S)=0$},\\
\dfrac{\cosh\ell(e;r)\cosh(e';r)-\cosh\ell(e_v^f; r)}{\sinh\ell(e;r)\sinh\ell(e';r)}, &\text{if $\chi(S)<0$},\\
\end{dcases}
\end{equation*}
where $e, e'$ are the edges in $f$ adjacent to $v$.
%%%

Let $a_v(r)$ be the \emph{cone angle} at the vertex $v$ with respect to $r$, which is the sum of each angle at $v$ in all triangles of $(T,\Theta)$ having $v$ as one of the vertices.
%%%
We define by $K_v(r):=2\pi-a_v(r)$ the \emph{curvature} at the vertex $v$ with respect to~$r$.
For $\chi(S)=0$, 
we easily check the invariance of the curvature under uniform scalings of the metric.
Moreover, the total curvature is always zero due to the Descartes-Gauss-Bonnet Theorem.
\begin{proposition}{\rm(\cite{CL}*{Proposition 3.1})} \label{DGB}
If $\chi(S)=0,$ then $\sum_{v\in V} K_v(r)=0$ holds for any~$r\in \R^{V}_{>0}$.
In addition, for a weighted triangulation of a closed ball $D$ in $\mathbb{E}^2$ and its metric~$r$,
if we define the cone angle $a_v(r)$  at the vertex $v$ as well as the case of $\chi(S)=0$ and define the curvature by 
\[
K_{D, v}(r)
:=\begin{dcases}
2\pi-a_v(r), & \text{if $v$ is in the interior of $D$},\\
\pi-a_v(r), & \text{if $v$ is on the boundary of $D$},
\end{dcases}
\]
then we have $\sum_{v} K_{D,v}(r)=2\pi$.
\end{proposition}
%%%%%%%%%%%%%%
In general, the curvature is bounded below by the function~$\phi$.
%%%%%%%%%%%%%%%%%%%%%%%%%%%%%%%%%%%%%%%%%%%%%%%%%%
\begin{proposition}{\rm (\cite{CL}*{Proposition 4.1}, \cite{Th}*{\S 13.7}, \cite{MR})}\label{topcurv}
If a sequence $\{r(n)\}_{n\in \N}$ in~$\R^{V}_{>0}$ converges on $\R^{V}_{\geq0}$ as $n\to \infty$,
then we have
\[
\lim_{n\to \infty}\sum_{v\in U} K_{Z, v}(r(n))=\phi(U),
\qquad
U:=\{v \in V_Z \ |\ \lim_{n \to \infty} r_v(n)=0\}.
\]
Furthermore, it holds for any nonempty proper subset $U$ of $V$ and $r \in \R^{V}_{>0}$ that
\[
\sum_{v\in U} K_{Z,v}(r)>\phi(U).
\]
\end{proposition}
%%%%%%%%%%%%%%%
A curvature uniquely determines a metric, up to a scalar multiple if $\chi(S)=0$.
This follows from the fact that the curvature can be regarded as the gradient of a convex function.
To see this, we define an injective map 
$X_S:\R^{V}_{>0} \to \R^V$  by
\[
X_S((r_v)_{v\in V}):=
\begin{dcases}
(\log r_v)_{v\in V}, & \text{if $\chi(S)=0$},\\
\left(\log\tanh\frac{r_v}{2}\right)_{v\in V}, & \text{if $\chi(S)<0$}.
\end{dcases}
\]
Let $R_S:X_S(\R^V_{>0}) \to \R^V_{>0}$ be the inverse map of $X_S$.
%%%%%%%%%%%%%%%%%%%%%%%%%%%%%%%%%%%%%%%%%%%%%%%%%%
\begin{proposition}{\rm (\cite{Th}, \cite{CL}*{Proposition 3.9, Corollary 3.11}, \cite{V}) } \label{convex}
%%%
There is a smooth convex function $\psi$ on $X_S(\R^V_{>0})$ such that $\nabla \psi=K \circ R_S$.
%%%
For $\chi(S)=0$, if we set 
\[
P:=\left\{r\in \R^{V}_{>0} \ |\ \prod_{v\in V} r_v =1\right\},\quad 
\]
then $\psi$ is strictly convex when restricted to $X_S(P)$ and $K$ is injective when restricted to~$P$.
%%%
On the other hand, if $\chi(S)<0$, then $\psi$ is strictly convex, consequently $K$ is injective.
\end{proposition}
%%%%%%%%%%%%%%%%%%%%%%%%%%%%%%%%%%%%%%%%%%%%%%%%%%

%%%%%%%%%%%%%%%%%%%%%%%%%%%%%%%%%%%%%%%%%%%%%%%%%%
%%%%%%%%%%%%%%%%%%%%%%%%%%%%%%%%%%%%%%%%%%%%%%%%%%
%%%%%%%%%%%%%%%%%%%%%%%%%%%%%%%%%%%%%%%%%%%%%%%%%%%%%%%%%%%%%%%%%%%%%%%%%%%%%%%%
%%%%%%%%%%%%%%%%%%%%%%%%%%%%%%%%%%%%%%%%%%%%%%%%%%%%%%%%%%%%%%%%%%%%%%%%%%%%%%%%
%%%%%%%%%%%%%%%%%%%%%%%%%%%%%%%%%%%%%%%%%%%%%%%%%%%%%%%%%%%%%%%%%%%%%%%%%%%%%%%%
\subsection{Combinatorial Ricci Flow}
The \emph{combinatorial Ricci flow} on a weighted triangulation~$(T,\Theta)$ of $S$ is a family of metrics $\{r(t)\}_{t\in[0,T)}$ that satisfies 
\begin{equation*}
\frac{d}{dt}r_v(t)=-K_v(r(t))\sigma_S(r_v(t)),
\quad
\sigma_S(r_v):=\begin{dcases}
r_v, & \text{if\ $\chi(S)=0$},\\
\sinh r_v, & \text{if\ $\chi(S)<0$}.
\end{dcases}
\end{equation*}
%%%
If we set $x(t):=X_S(r(t))$,
then we have
\[
\frac{d}{dt}x_v(t)=-K_v(r(t))=-\nabla \psi(x(t)), 
\]
where $\psi$ is given in Proposition~\ref{convex}.
%%%
Thus the combinatorial flow can be regarded as the gradient flow of a convex function 
and the combinatorial Ricci flow exists for all time (see also \cite{CL}*{Proposition 3.4}).
%%%
Moreover, for $\chi(S)=0$, the product $\prod_{v\in V} r_v(t)$ is constant in  $t\geq0$ since we deduce from Proposition \ref{DGB} that 
%%%%%%%%%%%%%%%%%%%%%%%%%%%%%%%%%%%%%%%%
\[
\frac{d}{dt}\prod_{v\in V} r_v(t)=-\left(\sum_{v\in V} K_v(t) \right)\prod_{v\in V} r_v(t)=0.
\]
%%%%%%%%%%%%%%%%%%%%%%%%%%%%%%%%%%%%%%%%

%%
%
As mentioned in Theorem \ref{CLT}, the combinatorial Ricci flow with an arbitrary initial data converges on $\R^{V}_{>0}$ at infinity if and only if (II) is valid.
Furthermore, Chow--Luo~\cite{CL}*{Theorems 1.1, 1.2} showed that
if the combinatorial Ricci flow converges on $\R^{V}_{>0}$ at infinity, 
then it converges exponentially fast to a circle pattern metric, where the decay rate depends on an initial data.
%%%

%%%%%%%%%%%%%
Let us conclude this section with a known result in the theory of gradient flows.
%%%%%%
\begin{proposition}{\rm(cf. \cite{AGS}*{Theorem 4.3.2})} \label{AGS}
Let $h:\R^n \to \R$ be a smooth convex function and
$\xi:[0,\infty)\to \R^n$ a gradient flow of $h$.
It holds for any $\tau>0$ and $\xi^\ast\in \R^n$ that 
\[
|\nabla h(\xi(\tau))|^2 \leq 
|\nabla h(\xi^\ast)|^2 +\frac{1}{\tau^2}|\xi^\ast-\xi(0)|^2. 
\] 
\end{proposition}
\begin{proof}
%%%
Since $\xi(t)$ is a gradient flow of $h$, we have $\frac{d}{dt}{\xi}(t)=-\nabla h(\xi(t))$,
and consequently
\begin{align*}
\frac{d}{dt} h(\xi(t))
&=\langle \nabla h(\xi(t)), \frac{d}{dt} \xi(t) \rangle
=-|\nabla h(\xi(t))|^2,\\
\frac{d}{dt}|\nabla h(\xi(t))|^2
&=2\langle D^2h(\xi(t)) \frac{d}{dt} \xi(t),\nabla h(\xi(t))\rangle
=-2\langle D^2h(\xi(t))\nabla h(\xi(t)),\nabla h(\xi(t))\rangle
\leq 0,
\end{align*}
where the inequality follows from the convexity of $h$.
%%%
Hence  $|\nabla h(\xi (t))|^2$  is nonincreasing, 
which leads to
\begin{align*}
\frac{\tau^2}2|\nabla h(\xi(\tau))|^2  
&\leq 
\int_0^\tau t|\nabla h(\xi(t))|^2 dt 
=-\int_0^\tau t \frac{d}{dt}h(\xi(t))  dt 
=-\tau h(\xi(\tau))  + \int_0^\tau h(\xi(t))  dt.
\end{align*}
Since the convexity of $h$ asserts
%%%%
\begin{align*}
h(\xi^\ast)-h(\xi(t))
\geq 
\langle \nabla h(\xi(t)), \xi^\ast- \xi(t) \rangle 
=\langle  - \frac{d}{dt} \xi(t), \xi^\ast-\xi(t)\rangle
=\frac12 \frac{d}{dt} |\xi^\ast-\xi(t)|^2,
\end{align*}
we have
\begin{align*}
\int_0^\tau h(\xi(t)) dt
&\leq
\int_0^\tau \left\{h(\xi^\ast)- \frac12 \frac{d}{dt} |\xi^\ast-\xi(t)|^2\right\} dt 
=\tau h(\xi^\ast)-\frac12 |\xi^\ast-\xi(\tau)|^2 +\frac12 |\xi^\ast-\xi(0)|^2.
\end{align*}
Combining this with the above inequality implies
\begin{align*}
 \frac{\tau^2}2|\nabla h(\xi(\tau))|^2  
\leq
\tau\{h(\xi^\ast)-h(\xi(\tau))\}-\frac12 |\xi^\ast-\xi(\tau)|^2 +\frac12 |\xi^\ast-\xi(0)|^2.
\end{align*}
%%%
We apply the convexity of $h$ again to have 
\begin{align*}
\tau\{h(\xi^\ast)-h(\xi(\tau))\} 
&\leq
\tau \langle \nabla h(\xi^\ast), \xi^\ast-\xi(\tau)\rangle %\\
%&\leq  \tau |\nabla h(\xi^\ast)| |\xi^\ast-\xi(\tau)| 
\leq \frac{\tau^2}{2} |\nabla h(\xi^\ast)|^2+\frac12|\xi^\ast-\xi(\tau)|^2.
\end{align*}
Substituting this into the above inequality and dividing with $2/\tau^2$ yields the proposition.
\end{proof}

%%%%%%%%%%%%%%%%%%%%
%%%%%%%%%%%%%%%%%%%%
%%%%%%%%%%%%%%%%%%%%
\section{Infinitesimal description of degenerate circle patten metric}
Throughout this section, we fix a weighted triangulation $(T,\Theta)$ of $S$ such that 
$\phi(U) \leq 0$ holds for any $U\subset V$ and $Z_T=\{z\in Z\ | Z\in\mathcal{Z}\} \neq \emptyset$.
Then the sets defined by
\begin{equation}\label{T0}
V_0:=V \setminus Z_T,\quad
E_0:=\{e\in E\ |\ e \cap Z_T =\emptyset\},\quad
%%%
F_0:=\{f \in F \ | f \cap Z_T =\emptyset \}\cup \{f_Z\}_{Z\in\mathcal{C}(Z_T)}
\end{equation}
and $\Theta_0:=\Theta|_{E_0}$ 
form a weighted cell-division $(T_0,\Theta_0)$ of $S$ into triangles and quadrangles satisfying (II) and (IV).
%%%%%%%%%%%%%%%%%%%%%%%%%%%%%%%%%%%%%%
For $n\geq 1$, let us inductively define $Z_{n-1}, Z_n,V_n $ by 
\begin{align*}
&
Z_{n-1}:=V\setminus \cup_{j=0}^{n-1} V_j,\quad
Z_n:=\{z\in Z\ |\ Z \subsetneq Z_{n-1}, Z \in \mathcal{Z}\}, \quad
V_n:=Z_{n-1} \setminus Z_n.
%%%
\end{align*}
We see that $Z_0=Z_T$.
There exists $N \in \N$ such that
\[
Z_N=\emptyset, \quad V=\cup_{n=0}^N V_n,
\]
and the family of sets  $\{V_n\}_{n=0}^N$ is mutually disjoint.
For each $v$, let $n(v)$ be a unique integer satisfying $v\in V_{n(v)}$.
Note that one vertex $v$ of $f\in F \cup \{f_Z\}_{Z\in\mathcal{Z}}$ is different from the remaining  vertices of $f$ if $n(v)\neq 0$. 
%%%%%%%%%%%%%%%%%%%%%%%%%%%%%%%%%%%%%%

For each $Z\in\mathcal{Z}\cup\{V\}$, define 
\[
Z':=\{z\in U\ |\ U\subsetneq Z,\ U \in \mathcal{Z}\}, \quad
V_Z:=Z\setminus Z',\quad
V^Z:=\{v\in V\ |\ \text{$v$ is a vertex of $f_Z$}\},
\]
where we set $V^V:=\emptyset$.
Then the correspondence from $\{V_Z\}_{Z\in\mathcal{Z}\cup\{V\}} $ to $\{U \in \mathcal{C}(V_n)\}_{0\leq n\leq N}$ is bijective.
It turns out that $V'=Z_T$ and $V_V=V_0$.

%%%%%%%%%%%%%%%%%%%%%%%%%%%%%%%%%%%%%%
%%%%%%%%%%%%%%%%%%%%%%%%%%%%%%%%%%%%%%
\subsection{Key lemmas}
Fix an arbitrary  $Z\in\mathcal{Z}$. 
Moreover for each  $U\in\mathcal{C}(Z')\cap \mathcal{Z}_4$, 
we fix a choice of a diagonal edge $e_{U}$ in $f_{U}$, and divide $f_{U}$ into two triangles $f_{U}^1$, $f_{U}^2$ by $e_{U}$.
We define 
\[
E_{Z}:=\{e \in E \ |\ e\cap Z \neq \emptyset, e\cap Z' =\emptyset\}
    \cup \{e_{U}\}_{U\in \mathcal{C}(Z') \cap \mathcal{Z}_4}\cup\{e^f_{v(f)}\}_{f\in \mathrm{Lk}(Z)}
\]
and $\Theta_Z:E_Z \to [0,\pi/2]$ by $\Theta_Z(e)=\Theta(e)$ for $e\in E$, and $0$ otherwise.
For $j=1,2,3$ and $v\in V_Z$, put
\begin{align*}
&\check{F}_Z^j:=\{f \in F \ |\   |f \cap V_Z|=j,\ f\cap Z'= \emptyset\},\quad 
\check{F}_Z^j(v):=\{f \in \check{F}_Z^j \ |\   v\in f\},\\
%%%%%
&\hat{F}_Z^j:=\{f\in\{f_{U}\}_{U\in\mathcal{C}(Z')\cap \mathcal{Z}_3} \cup \{f_{U}^1,f_{U}^2\}_{U\in \mathcal{C}(Z')\cap \mathcal{Z}_4} |\  |f \cap V_Z |=j\}, \quad 
\hat{F}_Z^j(v):=\{f\in\hat{F}_Z^j |\  v\in f\}, \\
&F_Z^j:=\check{F}_Z^j \cup \hat{F}_Z^j,\quad F_Z^j(v):=\{f\in F_Z^j \ |\ v\in f \}.
\end{align*}
%%%
%
By definition of $\mathcal{Z}$,
there is a continuous map $c$ from a closed ball $D$ in $\R^2$ to the closure of~$f_Z$ so that 
the restriction of $c$ to the interior of $D$ is a homeomorphism to $f_Z$.
The preimages of 
\begin{align*}
V_D:=V_Z \cup V^Z, \quad
E_D:=E_Z,\quad
F_D:=\cup_{j=1}^3 F_Z^j
\end{align*}
under $c$ with $\Theta_D$ satisfying $\Theta_D \circ c=\Theta_Z$ determine a weighted triangulation $(T_D, \Theta_D)$ of $D$. 

In order to give an infinitesimal description of degenerate circle patten metrics,
let us define the curvature-like function $K_Z:\R^{V}_{>0} \to \R^{V_Z}$ as follows.
Since $\mathbb{H}^2$ is infinitesimally identified with $\R^2$, regardless of the Euler characteristic of $S$,
we define the length $L(e;r)$ of $e \in E_Z$ with endpoints $u,v\in V_Z$ with respect to $r\in \R^{V}_{>0}$ by
\[
L(e;r):=\sqrt{r_v^2+r_u^2+2r_vr_u \cos \Theta(e)}
\]
as well as $\chi(S)=0$.
Similarly, the angle $\theta_v^{f,3}(r)$ at $v$ in $f\in F_Z^3$ with respect to $r$ is defined~by
\[
\theta_{v}^{f,3}(r):
=\arccos\left(\frac{L(e;r)^2+L(e';r)^2-L(e_v^f; r)^2}{2L(e;r)L(e';r)}\right),
\]
where $e, e'$ are the edges in $f$ adjacent to $v$.
Moreover, given $f\in F_Z^2$ with vertices $v,u \in V_Z$ and $w\in V^Z$, 
define the angle $\theta_v^{f,2}(r)$ at $v$ in $f$ with respect to $r$  by
\[
\theta_{v}^{f,2}(r)
:=\arccos\left(\frac{r_{v}\cos \Theta_Z(e^f_u)-r_{u}\cos \Theta_Z(e^f_v)}{L(e^f_w; r)}\right).
\]
%%%%
Finally, define the angle $\theta_v^{f,1}(r)$ at $v\in V_Z$ in $f \in F_Z^1$ with respect to $r$  by
\[
\theta^{f,1}_v(r):=\pi-\Theta_Z(e^f_v).
\]
We define $K_Z:\R^{V}_{>0} \to \R^{V_Z}$ by
\[
K_{Z,v}(r):=
2\pi
-\sum_{j=1}^3\sum_{f\in  F_Z^j(v)} \theta^{f,j}_v(r),
\]
which is invariant under uniform scalings of the metric.
%%%%%%%%%%%%%%%%%%%%%%%%%%%%%%%%%%%%%%%%%%%%%%%%%%%%%%%%%%%%%%%%%%%%

We first prove that $K_{Z}$ is independent of the choice of diagonal edges.
In what follows, the adjacency, connectivity of vertices and the link of a subset of vertices are respectively defined with respect to $(T,\Theta)$, not $(T_D, \Theta_D)$.
%%%%%%%%%%%%%%%%%%%%%%%%%%%%
\begin{proposition} \label{diagonal}
Fix $v \in V_Z$ and take $U\in \mathcal{C}(Z')\cap \mathcal{Z}_4$ such that $v\in f_{U}$.
%%%
Let $v_1,v_2,u$ be the vertices of $f_{U}$ other than $v$ such that $v_1, v_2$ are adjacent to $v$.
%%%
We divide $f_{U}$ into the two triangles $f_1, f_2$ $($resp.\ $f, f'$$)$ by the diagonal edge in $f_{U}$ not intersecting $v_i$ $($resp.\ $v$$)$,
where $v_i \in f_i$  for $i=1,2$ $($resp.\ $v\in f$$)$.
Then it holds for any $r\in \R^V_{>0}$ that 
\[
\theta_v^{f,j}(r)=\theta_v^{f_1, j_1}(r)+\theta^{f_2, j_2}(r),
\]
where $j,j_1, j_2 \in\{1,2,3\}$  satisfy  $f\in \hat{F}^j_Z$ and $f_i \in \hat{F}^{j_i}_Z$ for $i=1,2$.
\end{proposition}
\begin{proof}
We first observe for $i=1,2$ that
\[
\Theta_Z(e^f_v)=\Theta_Z(e^{f_i}_{v_i})=0,\quad
\Theta_Z(e^{f_i}_v)=\Theta_Z(e^{f_i}_u)=\Theta_Z(e^f_{v_i})=\frac{\pi}{2}.
\]
In the case of $j=1$, we see that $\theta^{f,1}_v(r)=\pi$.
If $u\in V_Z$ (resp.\ $u\in V^Z$), then we find $j_i=2$ (resp. $j_i=1$) and 
\[
\cos \theta^{f_i,2}_v(r)
=\frac{r_v \cos\Theta_Z(e_u^{f_i})-r_u \cos\Theta_Z(e_v^{f_i})}{L(e_{v_i}^{f_i};r)}
=0
\quad
\left(\text{resp.}\ \theta^{f_i,1}_v(r)=\frac{\pi}{2}\right).
\]

Assume $j=2$ and $v_1 \in V_Z$. It turns out that 
\[
\cos \theta^{f,2}_v(r)
=\frac{r_v \cos\Theta_Z(e_{v_1}^{f})-r_{v_1} \cos\Theta_Z(e_v^{f})}{L(e_{v_2}^f;r)}
=\frac{-r_{v_1}}{\sqrt{r_v^2+r_{v_1}^2}}.
\]
If moreover $u\in V_Z$ $($resp.\ $u\in V^Z)$, then $(j_1,j_2)=(3,2)$ (resp.\ $(j_1,j_2)=(2,1)$) and
\begin{gather*}
\cos \theta^{f_1,3}_v(r)
=\frac{L(e^{f_1}_{v_1};r)^2+L(e^{f_1}_{u};r)^2-L(e^{f_1}_{v};r)^2}{2L(e^{f_1}_{v_1};r)L(e^{f_1}_{u};r)}
=\frac{r_v}{\sqrt{r_v^2+r_{v_1}^2}}, \quad \cos \theta^{f_2,2}_v(r)=0\\
\left(\text{resp.}\
\cos \theta^{f_1,2}_v(r)=\frac{r_v \cos\Theta_Z(e_{v_1}^{f_1})-r_{v_1} \cos\Theta_Z(e_v^{f_1})}{L(e_{u}^{f_1};r)}
=\frac{r_v}{\sqrt{r_v^2+r_{v_1}^2}}, \quad
\theta^{f_2,1}_v(r)=\frac{\pi}{2}
\right).
\end{gather*}

Finally, if $j=3$, then we find
\[
\cos \theta^{f,3}(r)
=\frac{L(e^{f}_{v_1};r)^2+L(e^{f}_{v_2};r)^2-L(e^{f}_{v};r)^2}{2L(e^{f}_{v_1};r)L(e^{f_1}_{v_2};r)}
=\frac{r_v^2-r_{v_1}r_{v_2}}{\sqrt{r_v^2+r_{v_1}^2}\sqrt{r_v^2+r_{v_2}^2}}.
\]
As well as the case of $j=1,2$, regardless of the position of $u$, 
we have
\[
\cos \theta^{f_i,j_i}_v(r)=\frac{r_v}{\sqrt{r_v^2+r_{v_i}^2}}.
\]
 These with the trigonometric addition formulae conclude the proposition.
\end{proof}
%%%%%%%%%%%%%%%%%%%%

Given any $v \in V_Z$ and $U\in \mathcal{C}(Z')\cap\mathcal{Z}_4$ such that $v\in f_{U}$,
we denote by $e^{f_{U}}_v$ the diagonal edge in $f_U$ not intersecting $v$.
Let $\hat{f}_U$ be the triangle obtained by dividing $f_{U}$ by $e^{f_{U}}_v$ and containing $v$.
For $j=1,2,3$, define
\begin{align*}
\bar{F}_Z^j(v)
&:=\{ f\in\{f_{U}\}_{U\in\mathcal{C}(Z')}  \ |\ |\{u \in f\cap V_Z \ |\ \text{$u$ is adjacent to $v$}\}|=j-1\}.
\end{align*}
Then $f_U \in\bar{F}_{Z}^j(v)$ is equivalent to $\hat{f}_U \in \hat{F}^j_{Z}(v)$,
and it is natural to define $\theta^{f_U,j}_v(r):=\theta^{\hat{f}_U,j}_v(r)$. 
%%%
Proposition~\ref{diagonal} asserts 
\[
K_{Z,v}(r)=
2\pi
-\sum_{j=1}^3 \sum_{f\in  \check{F}_Z^j(v) \cup\bar{F}_Z^j(v)} \theta^{f,j}_v(r).
\]

%%%%%%%%%%%%%%%%%%%%%%%%%%%%%%%%%%%%%%%%%%%%%%%%%
\begin{definition}
A smooth curve $r(t):[0,\infty) \to \R^V_{>0}$ {\it degenerates of order $r\in \R^V_{>0}$}
if it converges on $\R^{V}_{\geq 0}$ at infinity and satisfies the following conditions:
\begin{itemize}
\item
We have $Z_T=\{z\in V\ |\  \lim_{t \to \infty} r_z(t)=0\}$ 
and $\lim_{t \to \infty} r_v(t)=r_v$ for any $v\in V_0$.
\item
For each $U\in\mathcal{Z}$, if we fix an arbitrary $v \in V_U$, 
then the limit of $\{({r_u(t)}/{r_{v}(t)})_{u\in V_U}\}_{t\geq0}$ at infinity exists and  the limit is parallel to the restriction $r|_{V_U}$ of $r$ to $V_U$.
\item
For any $U\in\mathcal{Z}\cup\{V\}$ and $v\in V_U, u\in U'$, we have $\lim_{t \to \infty} {r_u(t)}/{r_{v}(t)}=0$.
\end{itemize}
%%%%
\end{definition}
%%%%%%%%%%%%%%%%%%%%%%%%%%%%%%%%%%%%%%%%%%%%%%%%%
\begin{proposition}\label{coin}
Fix $v\in V_Z, j=1,2,3$ and consider a degenerate curve $\{r(t)\}_{t\geq0}$ of order~$r\in \R^V_{>0}$.
For $f\in \check{F}_Z^j(v)$ and  $U\in\mathcal{C}(Z')$ such that $f_{U} \in\bar{F}_Z^j(v)$, we have
%%%%%%
\begin{align}\label{out}
\lim_{t \to \infty} \theta_v^f(r(t))&=\theta^{f,j}_v(r), \\\label{in}
%%%%%
\lim_{t \to \infty} \sum_{v\in f'\in F, f' \subset f_{U}} \theta_v^{f'}(r(t))&=\theta^{f_U, j}_v(r).
\end{align}
\end{proposition}
%%%%%%%
\begin{proof}
Let $u, w$ be the vertices of $f$ other than $v$.
%
%Then we have $u,w \in V_Z \cup V^Z$ and $u=w$ may happen.
We can assume that $n(w)\leq n(u) \leq n(v)$ without loss of generality.
Note that $n(v) \neq 0$.
For $x=u,w$, there exists $\rho_x \in \R_{\geq 0}$ such that 
\[
\lim_{t \to \infty} \frac{r_v(t)}{r_x(t)}=\rho_{x}.
\]
We have $\rho_x >0$ if and only if $n(x)=n(v)$, in which case $\rho_x r_x=r_v$ holds.
The condition $j=1$ is equivalent to $n(w)\leq n(u)<n(v)$.
We similarly see  that 
$j=2$ (resp. $j=3$) if and only if $n(w)< n(u)=n(v)$ (resp. $n(w)=n(u)=n(v)$).

Let us define
\begin{align*}
N_E(t):&=r_v(t)^2-r_u(t) r_w(t) \cos\Theta(e^f_v)+r_v(t)\sum_{(x,y)=(u,w),(w,u)} r_x(t) \cos\Theta_Z(e^f_y), \\
%D_E(t):&=L(e^f_u;r(t))\cdot L(e^f_w;r(t)),\\
N_H(t):&=\cosh r_u(t)\cosh r_w(t) \sinh^2 r_v(t)-\sinh r_u(t)\sinh r_w(t)\cos\Theta_Z(e^f_v)\\
       &+\cosh r_v(t)\sinh r_v(t) \sum_{(x,y)=(u,w),(w,u)} \cosh r_x(t) \sinh r_y(t)\cos\Theta_Z(e^f_x)\\
       &+\sinh^2 r_v(t) \sinh r_u(t)\sinh r_w(t)\cos\Theta_Z(e^f_u)\cos\Theta_Z(e^f_w),\\
S_{x,y}(t):&=\left\{\left(\cosh r_v(t)\cosh r_x(t)+\sinh r_v(t)\sinh r_x(t)\cos\Theta(e^f_y;r(t))\right)^{2} -1\right\}^{1/2},
%D_H(t):&=S_{u,w}(t)S_{w,u}(t).
\end{align*}
where $(x,y)=(u,w),(w,u)$.
We then have
\[
\cos\theta^{f}_v(r(t))
=\begin{dcases}
{N_E(t)}/L(e^f_u;r(t)) \cdot L(e^f_w;r(t)), & \text{if $\chi(S)=0$,}\\
{N_H(t)}/S_{u,w}(t)\cdot S_{w,u}(t), & \text{if $\chi(S)<0$.}
\end{dcases}
\]
For  $(x,y)=(u,w),(w,u)$, we observe that 
\begin{align*}
%%%
&\lim_{t \to \infty}\frac{L(e^f_y;r(t))}{r_x(t)}
=\begin{dcases}
1, & \text{if $n(x)<n(v)$},\\
{L(e_y^f;r)}/{r_x}, & \text{if $n(x)=n(v)$},
\end{dcases}\\
&
\lim_{t \to \infty}  \frac{\sinh r_v(t)}{r_x(t)}
=\rho_x,\quad
\lim_{t \to \infty}  \cosh r_v(t)=1, \quad
\lim_{t \to \infty}\frac{\sinh r_x(t)}{r_x(t)}
=\begin{dcases}
{\sinh r_x}/{r_x}, & \text{if $n(x)=0$},\\
1, & \text{if $n(x)>0$},
\end{dcases}\\
%%%%%
&\lim_{t \to \infty}\frac{S_{x,y}(t)}{r_x(t)}
=\begin{dcases}
{\sinh r_x}/{ r_x}, & \text{if $n(x)=0$},\\
1, & \text{if $n(x)<n(v)$},\\
{ L(e_y^f;r)}/{r_x}, & \text{if $n(x)=n(v)$}.
\end{dcases}
%%%%
%%%%
\end{align*}
Substituting this into 
\[
\lim_{t\to\infty}\cos\theta^{f}_v(r(t))=\lim_{t\to\infty} \cos\theta^{f}_v(r(t)) \cdot \frac{r_u(t)r_w(t)}{r_u(t)r_w(t)}
\]
completes the proof of \eqref{out}.

To prove \eqref{in}, let us classify triangles $f'$ satisfying $v\in f'\in F, f' \subset f_{U}$ into
\[
F_{U}^{j'}(v):=\{f'\in F\ |\ v\in f', \ | f'\cap U|=j'  \}
\]
for ${j'}=1,2$, which asserts 
\[
\sum_{v\in f'\in F, f' \subset f_{U}} \theta_v^{f'}(r(t))
=
\sum_{f'\in F_{U}^1(v)} \theta_v^{f'}(r(t))+\sum_{f'\in F_{U}^2(v)} \theta_v^{f'}(r(t)).
\]
It turns out that $F_{U}^1(v)\subset \mathrm{Lk}(U)$ and $F_{U}^1(v)$ consists of two elements $f_1, f_2$.
For $f'\in F_{U}^2(v)$, the vertices $u',w'$ in $f'$ other than $v$ are in $U$. 
Since we find for $(x,y)=(u',w'),(w',u')$ that 
\begin{gather*}
\lim_{t \to \infty} \frac{r_x(t)}{r_v(t)}=0, \quad
\lim_{t \to \infty}\frac{L(e^{f'}_x;r(t))}{r_v(t)}=1, \quad
\lim_{t \to \infty}\frac{S_{x,y}(t)}{r_v(t)}=1,
%%%%
\end{gather*}
a similar argument to \eqref{out} yields
\begin{align*}
\lim_{t \to \infty} \cos \theta^{f'}_v(r(t))  
=\lim_{t \to \infty} \cos \theta^{f'}_v(r(t)) \cdot \frac{r_v(r)^2}{r_v(t)^2}=1,
\end{align*}
namely, $\lim_{t \to \infty} \theta^{f'}_v(r(t))=0$.
%%%%%%%%%%%%%%%%%%%%%%%%%%%%%%%%%%%%%%%%%%%%%%%%%%%%%%%%%%%%%%%%%%
We therefore conclude
\begin{equation}\label{red1}
\lim_{t\to \infty}\sum_{v\in f'\in F, f' \subset f_{U}} \theta_v^{f'}(r(t))
=
\lim_{t\to \infty} \left\{\theta_v^{f_1}(r(t))+\theta_v^{f_2}(r(t))\right\}.
\end{equation}
%%%
Let us set $\hat{f}_{U}$ as $f_{U}$ itself if $U \in \mathcal{C}(Z') \cap \mathcal{Z}_3$, 
and as the triangle obtained by dividing $f_{U}$ by~$e^{f_{U}}_v$ and containing $v$ if $U\in\mathcal{C}(Z')\cap \mathcal{Z}_4$.
%%%%
%
For $i=1,2$, let $u_i, w_i$ be the vertices of $f_i\in F_{U}^1(v)$ other than $v$ such that $u_i\in U$ and $w_i \in V_Z \cup V^Z$, hence $n(w_i)\leq n(v)<n(u_i)$.
Then $w_1, w_2$ become the vertices of $\hat{f}_{U}$, and $e_{u_i}^{f_i}=e_{w_k}^{\hat{f}_{U}}$ for $(i,k)=(1,2),(2,1)$.
By $U\in \mathcal{Z}$, we have 
\begin{equation}\label{equal}
\Theta_Z(e^{\hat{f}_{U}}_v)=\pi-(\Theta_Z(e^{\hat{f}_U}_{w_1})+\Theta_Z(e^{\hat{f}_U}_{w_2})).
\end{equation}
Arguing similarly for \eqref{out}, we have for $(i,k)=(1,2),(2,1)$ that 
\[
\lim_{t\to \infty} \cos \theta^{f_i}_v (r(t))
=\lim_{t\to \infty} \cos \theta^{f_i}_v (r(t)) \cdot\frac{r_{v}(t)r_{w_i}(t)}{r_{v}(t)r_{w_i}(t)}
=
\begin{dcases}
\cos\Theta_Z(e_{w_k}^{\hat{f}_{U}}), & \text{if $n(w_i)<n(v)$},\\
\frac{r_v+r_{w_i}\cos\Theta_Z(e_{w_k}^{\hat{f}_{U}})}{L(e_{w_k}^{\hat{f}_{U}};r)}, & \text{if $n(w_i)=n(v)$.}
\end{dcases}
\]
Since the condition $j=1$ is equivalent to 
$n(w_1),n(w_2)<n(v)$, the trigonometric addition formulae with \eqref{equal} imply \eqref{in}.
%%%
Similarly, the condition $j=2$ (resp. $j=3$) is equivalent to 
$\min\{n(w_1), n(w_2)\}<\max\{n(w_1), n(w_2)\}=n(v)$ (resp. $n(w_1)=n(w_2)=n(v)$), 
and we have the desired results.
\end{proof}
The relations \eqref{out}\eqref{in} hold true for $U\in\mathcal{C}(V')=\mathcal{C}(V_0)$.
%%%%%%%%%%%%%%%%%%%%%%%%%%%%%%%%%%%
\begin{proposition}\label{coin0}
Fix $v\in V_0$ and consider a degenerate curve $\{r(t)\}_{t\geq0}$ of order $r\in \R^V_{>0}$.
%%%%%%
For $f\in F\cap F_0$ with $v\in f$ and $U\in\mathcal{C}(V_0)$ with $v\in f_U$, we have
\begin{align} \label{out0}
\lim_{t \to \infty} \theta^{f}_v(r(t))=\theta^{f}_v(r), \\ \label{in0}
\lim_{t \to \infty} \sum_{v\in f'\in F, f' \subset f_{U}} \theta_v^{f'}(r(t))
=\theta^{f_U}_v(r),
\end{align}
where we use the notation $\theta^{\tilde{f}}_v$ for the sum of angles at $v$ in $\tilde{f}$ 
if $\tilde{f}$ contains an edge whose both endpoints are $v$.
\end{proposition}
%%%%%%
\begin{proof}
Suppose that all the vertices of $f,f_U$ are different from each other.
The relation~\eqref{out0}  follows from the continuity of the trigonometric and hyperbolic functions.
To prove \eqref{in0}, we use the same notations in the proof of \eqref{in}.
%%%
Then \eqref{red1} and \eqref{equal} hold.
By~$V^V=\emptyset$, we have  $n(w_i)=n(v)=0<n(u_i)$ and 
for $(i,k)=(1,2),(2,1)$ that
\begin{align*}
\lim_{t\to \infty}\cos\theta^{f_i}_v(r(t))
=
\begin{dcases}
\frac{r_v+r_{w_i} \cos\Theta_Z(e_{w_k}^{\hat{f}_{U}})}{\ell(e_{w_k}^{\hat{f}_{U}};r)}, &\text{if $\chi(S)=0$},\\
\frac{\cosh r_{w_i}\sinh r_{v}+\cosh r_{v}\sinh r_{w_i}\cos\Theta_Z(e_{w_k}^{\hat{f}_{U}})}{\sinh\ell(e_{w_k}^{\hat{f}_{U}};r)},  &\text{if $\chi(S)<0$}.
\end{dcases}
\end{align*}
This and the trigonometric addition formulae with \eqref{equal} complete the proof of \eqref{in0}.

The case that some vertices of $f,f_U$ coincide is proved similarly.
\end{proof}
%%%%%%%%%%%%%%%%%%%%%%%%%%%%%%%%%%%%%%
\begin{corollary}\label{outin0}
Let $\{r(t)\}_{t\geq0}$ be a degenerate curve of order $r$. It holds for $v\in V_Z$ that 
\begin{equation}\label{Z}
\lim_{t \to \infty} K_{v}(r(t))=K_{Z,v}(r).
\end{equation}
For the curvature $K_V$ on the weighted cell-division $(T_0,\Theta_0)$ of $S$, 
we have for $v\in V_0$ that 
\begin{equation}\label{0}
\lim_{t \to \infty} K_{v}(r(t))=K_{V,v}(r).
\end{equation}
\end{corollary}
\begin{proof}
A direct computation provides
\begin{align*}
&K_{v}(r)-K_{Z,v}(r(t)) \\
&=\sum_{j=1}^3
\left\{
\sum_{f\in \check{F}^j_Z(v)} \left(\theta^{f,j}_v(r)-\theta_v^f(r(t))\right)
+
\sum_{{f}_{U}\in\bar{F}_Z^j(v)}\left(\theta^{f_U,j}_v(r)-\sum_{ v\in f\in F, f \subset f_{U}} \theta_v^f(r(t))\right)
\right\}.
%%%%%%%%
\end{align*}
Applying Proposition \ref{coin} leads to \eqref{Z}.
Similarly, \eqref{0} follows from Proposition \ref{coin0}.
\end{proof}

For $U\in \mathcal{Z}\cup\{V\}$, since $K_{U}(r)$ depends only on $r|_{V_U}$, we can regard $K_U:\R^{V_U}_{>0} \to \R^{V_U}$.
The following corollary immediately follows from the proofs of Propositions~\ref{coin},~\ref{coin0}.
\begin{corollary}\label{outin}
Given $U\in\mathcal{Z}\cup\{V\}$ and $v\in V_U$, 
we assume that a curve $\{r(t)\}_{t\geq0}$ satisfies  
\[
\lim_{t \to \infty} \frac{r_u(t)}{r_{v}(t)}=0, \quad
\lim_{t \to \infty} \frac{r_v(t)}{r_{w}(t)}=0
\]
for any $u\in U', w\in V^U$.
%%%
Moreover for any $u\in V_U$, if the limit 
\[
\rho_u:=
\begin{dcases}
\lim_{t\to\infty} \frac{r_u(t)}{r_{v}(t)}, &\text{if $U\in \mathcal{Z}$},\\
\lim_{t\to\infty} r_u(t),  &\text{if $U=V$}
\end{dcases}
\]
exists and $\rho_u>0$,
then we have
\[
\lim_{t \to \infty} K_{u}(r(t))=K_{U,u}((\rho_x)_{x\in V_U}).
\]
\end{corollary}

%%%%%%%%%%%%%%%%%%%%%%%%%%%%%%%%%%%%%%
We will show that $(T_D,\Theta_D)$ and $K_Z$ satisfy a series of the following lemmas in a similar way to the proof of the results for $(T,\Theta)$ and $K$ mentioned in Section~2.
\begin{lemma}\label{likeDGB}
It holds for any $r\in \R^V_{>0}$ that
$
\sum_{v\in V_Z} K_{Z,v}(r)=0.
$
\end{lemma}
\begin{proof}
For a degenerate curve $\{r(t)\}_{t\geq0}$ of order $r$, Proposition \ref{DGB} and Corollary \ref{outin0} yield 
\begin{equation}\label{like}
\sum_{v\in V_Z}K_{Z,v}(r)+\lim_{t\to \infty}\sum_{w\in V^Z} K_{D,v}(r(t))
=\lim_{t\to \infty}\sum_{v\in V_D} K_{D,v}(r(t))
=2\pi,
\end{equation}
where we  use the same notations $V_D, V^Z$ for their preimages under a continuous map $c$ from~$D$ to the closure of $f_Z$.
In order to analyze the behavior of $\sum_{w\in V^Z} K_{D,w}(r(t))$, define for $j=1,2,3$
\[
F_D^j:=\{f \in F_D \ |\  \text{$j$ vertices of $f$ are in $V^Z$}  \}.
\]
%%%
Without loss of generality, we can choose diagonal edges $\{e_{U}\}_{U\in \mathcal{C}(Z')\cap \mathcal{Z}_4}$ such that $F_D^3=\emptyset$.
Then $F^j_D=F^{3-j}_Z$ for $j=1,2$.
%%%%
Since $\lim_{t\to \infty}  \theta^f_w(r(t))=0$ holds for $w\in V_Z, w\in f\in F_D^1$ as well as \eqref{in}, 
we have
\begin{align*}
&\lim_{t\to \infty} \sum_{w\in V^Z}K_{D,w}(r(t)) 
=|\mathrm{Lk}(Z)|\pi-\lim_{t\to \infty}a(r(t)),\\
&a(r):=\left(\text{the sum of each Euclidean angles at $V^Z$ of all triangles in $F_D^2$ with respect to $r$}\right).
\end{align*}
The correspondence from $f \in \mathrm{Lk}(Z)$ to $f'\in F_D^2$ containing $e^f_{v(f)}$ is bijective,  
and the sum of angles at the endpoints of $e^f_{v(f)}$ in $f'$ with respect to $r(t)$ converges to $\Theta_Z(e^f_{v(f)})=\Theta(e^f_{v(f)})$ at infinity by a similar argument to \eqref{out}.
%%%
Thus we see that
\[
\lim_{t\to \infty}a(r(t))=\sum_{f\in\mathrm{Lk}(Z)} \Theta(e^f_{v(f)})=(|\mathrm{Lk}(Z)|-2)\pi
\]
by Proposition \ref{CW}.
Substituting this into \eqref{like} completes the proof of the lemma.
\end{proof}
%%%%%%%%%%%%%%%%%%%%%%%%%%%%%%%%%%%%%%
For $U \subset V_Z$, we define  
\[
\mathrm{Lk}_Z(U):=\cup_{v\in U}\{f\in F_D\ |\ \text{$v\in f$ and there exists $e\in E_D$ such that $e\subset f$ and $e\cap U=\emptyset$}\}
\]
and $\tau_{U}^Z$ by the CW-subcomplex of $T_D$, consisting of all cells whose vertices are contained in~$U$.
We consider the function defined for $U\subset V_Z$ by
\[
\phi_Z(U):=-\sum_{f\in \mathrm{Lk}_Z(U)} (\pi-\Theta_Z(e^f_{v(f)})) +2\pi \chi(\tau_{U}^Z).
\]
%%
%%%%%%%%%%%%%%%%%%%%%%%%%%%%%%%
\begin{lemma}\label{less}
It holds for any nonempty proper subset $U$ of $V_Z$ that $\phi_Z(U)<0$.
\end{lemma}
%%%%%%%%%%%%%%%%%%%%%%%%%%%%%%%
\begin{proof}
Let $U$ be the nonempty subset of $V_Z$  of the smallest cardinality such that $\phi_Z(U)\geq 0$.
The additivity of the right-hand side of $\phi_Z(U)$ with respect to connected components of $\tau_U^Z$ 
leads to the connectivity of $\tau_U^Z$ in $T_D$.
By the theory of planar graphs, we have $\chi(\tau_U^Z) \leq 1$ with equality if and only if $\tau_{U}^Z$ is contractible. 
%%%
If $\mathrm{Lk}_{Z}(U)= \emptyset$,
then any two vertices in $V_D\setminus U$ are not joined by $e\in E_D$, which contradicts  $V^Z \subset V_D\setminus U$.
Hence $\mathrm{Lk}_Z(U)\neq \emptyset$.
It follows from the fact $\Theta_Z(e) \leq \pi/2$ that  
\[
0\leq \phi_{Z}(U)
=-\sum_{f \in \mathrm{Lk}_{Z}(U)}(\pi-\Theta_Z(e^f))+2\pi\chi(\tau_{U}^Z) \leq -|\mathrm{Lk}_{Z}(U)|\cdot\frac{\pi}2+2\pi\chi(\tau_{U}^Z),
\]
implying $\chi(\tau_{U}^Z)=1,
|\mathrm{Lk}_Z(U)|\leq 4$ and $\sum_{f\in \mathrm{Lk}_{Z}(U)} \Theta_Z(e^f_{v(f)}) \geq (|\mathrm{Lk}_Z(U)|-2)\pi$.
In this case $\{e^f_{v(f)}\}_{f\in \mathrm{Lk}_{Z}(U)}$ form a null-homotopic loop 
hence $|\mathrm{Lk}_Z(U)|\in\{3,4\}$ by the assumption of triangulations.
%%%
%%%%
Moreover, we see that $\{e^f_{v(f)}\}_{f\in \mathrm{Lk}_{Z}(U)} \subset E$, 
and there exists $W\subset Z'$  such that $\mathcal{C}(W) \subset \mathcal{C}(Z')$,
$\tau_{U\cup W}$ is contractible and $\mathrm{Lk}(U\cup W)=\mathrm{Lk}_Z(U)$.
Therefore we have $\phi(U\cup W)\geq 0$, namely $U \cup W  \in\mathcal{Z}$, implying $U=V_Z$ due to $U\subset V_Z=Z\setminus Z'$.
\end{proof}
%%%%%%%%%%%%%%%%%%%%%%%%%%%%%%%
%%%%
%%%%%%%
\begin{lemma}\label{compati}
If a sequence $\{r(n)\}_{n\in \N}$ in\ $\R^{V_Z}_{>0}$ converges on $[0,\infty]^{V_Z}$ as $n\to \infty$,
then we have
\[
\lim_{n\to \infty}\sum_{v\in U} K_{Z,v}(r(n))=\phi_Z(U),
\qquad
U:=\{v \in V_Z \ |\ \lim_{n \to \infty} r_v(n)=0\}.
\]
Furthermore, it holds for any nonempty proper subset $U$ of $V_Z$ and $r \in \R^{V_Z}_{>0}$ that
\[
\sum_{v\in U} K_{Z,v}(r)>\phi_Z(U).
\]
\end{lemma}
%%%%%%%%%%%%%%%%%
\begin{proof}
Let $U$ be a subset of $V_Z$, possibly $U=\emptyset, V_Z$.
By an argument similar to $\tau_{U}$, we find
\[
\phi_Z(U)=-\sum_{f\in\mathrm{Lk}_{Z}(U)}(\pi -\Theta_Z(e^f_{v(f)}))+2\pi|U|-\pi(|F_{U}^2|+|F_{U}^3|),
\quad
F_{U}^j:=\{f\in F_D \ |\  |f \cap U | =j\},
\]
and $F_U^1=\mathrm{Lk}_Z(U)$.
It turns out that 
\begin{align*} 
2\pi|U|-\sum_{v\in U}K_{Z,v}(r) 
&=
\sum_{v\in U}\sum_{j=1}^3
\sum_{f\in  F_{U}^1\cap F_Z^j(v)} \theta^{f,j}_v(r)\\
%%%%%%%%
&
+\sum_{v_1,v_2\in U, v_1\neq v_2}
\sum_{j=2}^3 \sum_{v_1,v_2\in f\in  F_{U}^2\cap F_Z^j} (\theta_{v_1}^{f,j}(r)+\theta_{v_2}^{f,j}(r))\\
%%%%%%
&
+\sum_{v_1,v_2,v_3\in U\text{:distinct}}
\sum_{v_1,v_2,v_3 \in f\in  F_{U}^3\cap F_Z^3}\left( \theta_{v_1}^{f,3}(r)+\theta_{v_2}^{f,3}(r)+\theta_{v_3}^{f,3}(r)\right).
%%%
\end{align*}
%%%%
Given $f\in  F_{U}^3\cap F^3_Z$ with vertices $v_1, v_2, v_3$,
and $\tilde{f}\in  F_{U}^2\cap F_Z^2$ with vertices $\tilde{v}_1, \tilde{v}_2$ in $U$, we find
\[
\theta^{f,3}_{v_1}(r(n))+\theta^{f,3}_{v_2}(r(n))+\theta^{f,3}_{v_3}(r(n))
=
\theta^{\tilde{f},2}_{\tilde{v}_1}(r(n))+\theta^{\tilde{f},2}_{\tilde{v}_2}(r(n))=\pi.
\]
%%%
Arguing similarly for Proposition \ref{coin},
we have for $f\in  F_{U}^2\cap  F_Z^3$ with vertices $v_1,v_2$ in $U$ that
\begin{equation}\label{23}
\lim_{n\to \infty}(\theta^{f,3}_{v_1}(r(n))+\theta^{f,3}_{v_2}(r(n)))=\pi
>\theta^{f,3}_{v_1}(r(n))+\theta^{f,3}_{v_2}(r(n)),
\end{equation}
where the inequality follows from Proposition \ref{angle_sym} for the case of $\chi(S)=0$.
Similarly, for $f\in  F_{U}^1\cap F_Z^j$ with $j=2,3$ having  $v\in U$ as one of the vertices, 
we notice
\begin{equation}\label{1323}
\lim_{n\to \infty}{\theta}^{f,j}_{v}(r(n))=\pi-\Theta_Z(e^f_v)>{\theta}^{f,j}_{v}(r(n)).
\end{equation}
%%%
This leads to
\[
2\pi|U|-\lim_{n\to\infty}\sum_{v\in U}K_{Z,v}(r(n))
=\sum_{f\in \mathrm{Lk}_Z(U)} ( \pi-\Theta_Z(e^f_{v(f)}))+\pi|F_{U}^2|+\pi|F_{U}^3|
=2\pi|U|-\phi_Z(U),
\]
proving the first claim.
Moreover, in the case of $|V_Z|\geq 2$, if $U\neq \emptyset, V_Z$, then the connectivity of $V_Z$ ensures that 
at least one of 
\[
F_{U}^1\cap  F_Z^2, \quad
F_{U}^1\cap  F_Z^3, \quad
F_{U}^2\cap  F_Z^3
\]
is nonempty, which implies the second claim by \eqref{23}, \eqref{1323}.
\end{proof}
%%%%%%%%%%%%%%%%%%%%%%%%%%%%%%%
\begin{remark}\label{310}
Arguing similarly to Lemma \ref{compati}, 
we find that the first claim in Proposition~\ref{topcurv} holds for a sequence in $\R^V_{>0}$ converging on $[0,\infty]^{V}$.
We moreover observe from the proof of Lemma \ref{compati} that, 
for a sequence $\{r(n)\}_{n\in \N}$ in\ $\R^{V}_{>0}$ converging on $[0,\infty]^{V}$ as $n\to \infty$,
we have
\[
\lim_{n\to \infty}\sum_{v\in U} K_{v}(r(n))=\phi(U)
\]
if and only if it holds for any $u\in U$ and $v\in  V\setminus U$ adjacent to $u$ that 
\[
%\lim_{n \to \infty} r_u(n)=0, \quad 
\lim_{n\to \infty} \frac{r_u(n)}{r_v(n)}=0,
\]
in addition  $\lim_{n \to \infty} r_u(n)=0$ if $\chi(S)<0$, 
where there is no restriction on a behavior of $\{r(n)\}_{n\in \N}$ at vertices 
which do not belong to $U$ nor are not adjacent to $U$.
\end{remark}
%%%%%%%%%%%%%%%%%%%%%%%%%%%%%%%
\begin{lemma}
%%%
There exists a smooth function $\psi_Z$ on $\R^{V_Z}$ such that $\nabla \psi_Z=K_Z \circ R_Z $,
where a map $R_Z: \R^{V_Z} \to \R^{V_Z}_{>0}$ sends $(x_v)_{v\in V_Z}$ to $(e^{x_v})_{v\in V_Z}$.
\end{lemma}
%%%%
\begin{proof}
Given $v,u \in V_Z$ and $f\in F^1_Z(v), \tilde{f}\in F^j_Z(v)$ with $j=2,3$, 
if $u\notin \tilde{f}$, then we  trivially have
\begin{equation}\label{111}
\frac{\partial (\theta_v^{f,1} \circ R_Z)}{\partial x_v}
=0,\quad
\frac{\partial (\theta_v^{f,1} \circ R_Z)}{\partial x_u}
=0,\quad
\frac{\partial (\theta_v^{\tilde{f},j} \circ R_Z)}{\partial x_u}
=0.
\end{equation}
For distinct vertices $v,u\in V_Z$ of $f\in  F^3_Z$,
Proposition~\ref{angle_sym} for $\chi(S)=0$ yields 
\begin{equation}\label{333}
\frac{\partial (\theta_v^{f,3} \circ R_Z)}{\partial x_u}
= \frac{\partial \theta_v^{f,3} }{\partial r_u}  R_{Z,u}
= \frac{\partial \theta_u^{f,3}}{\partial r_v} R_{Z,v}
=\frac{\partial (\theta_u^{f,3} \circ R_Z)}{\partial x_v}>0
>\frac{\partial (\theta_v^{f,3} \circ R_Z)}{\partial x_v}.
\end{equation}
%%%%%%%
Similarly, we find for distinct vertices $v,u\in V_Z$ of $f\in F^2_Z$ that
\begin{equation}\label{222}
\frac{\partial (\theta_v^{f,2} \circ R_Z)}{\partial x_u}
=\frac{\partial (\theta_u^{f,2} \circ R_Z)}{\partial x_v}
=-\frac{\partial (\theta_v^{f,2} \circ R_Z)}{\partial x_v}
\geq 0,
\end{equation}
with equality in the inequality if and only if $\Theta_Z(e^f_v)=\Theta_Z(e^f_u)=\pi/2$.
%%%
Thus we conclude 
\[
\frac{\partial (K_{Z,v} \circ R_Z)}{\partial x_u}=
\frac{\partial (K_{Z,u} \circ R_Z)}{\partial x_v}
\]
and the Poincar\'e lemma together with the regularity of $K_{Z}$ gives the desired $\psi_Z$.
\end{proof}
%%%%%%%%%%%%%%%%%%%%%%%%%%%%%%%%%%%%%%%%%%%%%%%%%%%%
%%%%%%%%%%%%%%%%%%%%%%%%%%%%%%%%%%%%%%%%%%%%%%%%%%%%
%%%%%%%%%%%%%%%%%%%%%%%%%%%%%%%%%%%%%%%%%%%%%%%%%%%%
The following fact of linear algebra plays a key role in the proof of the convexity of $\psi_Z$.
\begin{proposition}{\rm (\cite{CL}*{Lemma 3.10})} \label{LA}
If a symmetric matrix $A=[a_{ij}]_{1\leq i,j \leq n}$ satisfies 
\[
a_{ii}>0 \geq a_{ij}, \quad \sum_{k=1}^n a_{ik}=0
\]
for any $1\leq i,j \leq n$ with $i\neq j$,
then $A$ is semi-positive definite and its kernel is one-dimensional with basis $(1,\ldots, 1)$.
\end{proposition}
%%%%%%%%%%%%%%%%%%%%%%%%%%%%%%%%%%%%%%%%%%%%%%%%%%%%
\begin{lemma}\label{injective}
The function $\psi_Z$ is convex.
Furthermore, $K_Z$ is injective when restricted to %the hyperplane
\[
P_Z:=\left\{ r \in \R^{V_Z}_{>0} \ |\ \prod_{v\in V_Z} r_v =1\right\}.
\]
\end{lemma}
%%%%%%%%%%%%%%%%%%%%%%%%%%%%%%%%%%%%%%%%%%%%%%%%%%%%
\begin{proof}
For distinct vertices $u,v\in V_Z$, we observe from \eqref{111}--\eqref{222} and Lemma~\ref{likeDGB} that 
\[
\frac{\partial (K_{Z,v} \circ R_Z)}{\partial x_v} \geq 0 \geq \frac{\partial (K_{Z,v} \circ R_Z)}{\partial x_u},\quad
\sum_{w\in V_Z}\frac{\partial (K_{Z,v} \circ R_Z)}{\partial x_w}=
\sum_{w\in V_Z}\frac{\partial (K_{Z,w} \circ R_Z)}{\partial x_v}=0.
\]
If there is no $v\in V_Z$ such that 
\begin{equation}\label{nabla}
\frac{\partial (K_{Z,v}\circ R_Z)}{\partial x_v}=0,
\end{equation}
then Proposition \ref{LA} leads to the convexity of $\psi_Z$.
%%%
The injectivity of $K_Z|_{P_Z}$ follows from the second order Taylor expansion of $\psi_Z$ .

%%%
If there exists $v\in V_Z$ satisfying \eqref{nabla},
then $|V_Z|=1$, consequently $K_Z|_{P_Z}$ is injective.
Indeed for such $v$, we find $F^3_Z(v)=\emptyset$ by \eqref{333}.
%%%
If there exists $f\in F^2_Z$ with vertices $v,u\in V_Z, w\in V^Z$,
then $\Theta_Z(e_v^f)=\Theta_Z(e_u^f)=\pi/2$ by~\eqref{222}.
%%%
It follows from $F^3_Z(v) =\emptyset$ that 
there exist $f'\in  F^2_Z$ and $w' \in V^Z$ such that $e_w^f\subset f', w' \in f'$.
%%%
Furthermore, there exist $\tilde{f}, \tilde{f}'\in  F_D$ such that $e_u^f\subset\tilde{f}, e_u^{f'}\subset \tilde{f'}$.
Let $\tilde{u}$ (resp.\ $\tilde{u}'$) be the vertex of $\tilde{f}$ (resp.\ $\tilde{f}'$) other than $v,w$ (resp.\ $v,w'$).
If $\tilde{f}=\tilde{f}'$, namely $\tilde{u}=w',\tilde{u}'=w$, then $\mathrm{Lk}_{Z}(\{v\})=\{f,f'\tilde{f}\}$ and 
\[
\phi_Z(\{v\})=-3\pi+ \Theta_Z(e^{f}_v)+\Theta_Z(e^{f'}_v)+\Theta_Z(e^{\tilde{f}}_v)+2\pi=\Theta_Z(e^{\tilde{f}}_v)\geq 0,
\]
which contradicts Lemma \ref{less}.
Arguing similarly, if $w,w'$ are joined by some  $e\in E_Z$,
then each of three edges $e,e^f_v, e^{f'}_v$ and $e,e^f_u, e^{f'}_u$
form a null-homotopic loop and the sum of the weights of the three edges is at least $\pi$. 
However at least one of the three edges do not form the boundary of an element in $F_D$, contradicting Lemma \ref{less}.
Thus $w,w'$ are not joined by $e\in E_Z$, implying $Z\in\mathcal{Z}_4$ and $\Theta(e)=\pi/2$ for $e\in \{e^f_{v(f)}\}_{f\in\mathrm{Lk}(Z)}$.
%%%%
Moreover, we have $\Theta_Z(e^{\tilde{f}}_v)=\Theta_Z(e^{\tilde{f}'}_v)=\pi/2$, regardless of the position of $\tilde{u},\tilde{u}'$ by $Z\in\mathcal{Z}_4$ and \eqref{222}.
We see that $\tilde{u}\neq \tilde{u}'$ since $\tilde{u}=\tilde{u}'$ implies $\phi_Z(\{v\})=0$.
If $\tilde{u}\in V^Z$, then there exists 
$e\in \{e^f_{v(f)}\}_{f\in\mathrm{Lk}(Z)}$ such that $e$ joins $w', \tilde{u}$.
We find that $e^{\tilde{f}}_v, e^f_{u}, e^{f'}_u, e$ form a null-homotopic loop and there exists $U  \subsetneq V_Z$ such that 
$\tilde{u}' \in U$, $\mathrm{Lk}_Z(U)=\{e^{\tilde{f}}_v, e^f_{u}, e^{f'}_u, e\}$ and $\phi_Z(U)=0$, which never happens by Lemma \ref{less}.
However if $\tilde{u}\in V_Z$, then by $F^3_Z(v)=\emptyset$,
there exists $\hat{f}\in F^2_Z$ such that $e^{\tilde{f}}_w \subset \hat{f}$.
For $\hat{w} \in \hat{f}\cap V^Z$, $\tilde{u}\neq  \tilde{u}'$ implies $\hat{w}\neq w'$, 
hence $\hat{w},w$ are joined by some $e\in\{e^{f}_{v(f)}\}_{f\in \mathrm{Lk}(Z)}$. 
The three edges $e^{\tilde{f}}_{\tilde{u}},e^{\hat{f}}_{\tilde{u}},e$ form a null-homotopic loop
and surround $\tilde{u}$, which leads to a contradiction as well as $\tilde{u} \in V^Z$.
%%%
Thus $F^2_Z(v)=\emptyset$, and $v$ is surrounded by elements in $F^1_Z$, that is, $V_Z=\{v\}$ by the construction of $Z$.
\end{proof}
%%%%%%%%%%%%%%%%%%%%%%%%%%%%%%%%%%%%%%
%%%%%%%%%%%%%%%%%%%%%%%%%%%%%%%%%%%%%%
\begin{lemma}{\rm (cf. \cite{Th} \cite{CL}*{\S6.1}, \cite{MR}*{\S7})} \label{origin}
%%%
There exists a unique $r\in P_Z$ such that $K_Z(r)=0$.
\end{lemma}
\begin{proof}
Define the hyperplane $Q_Z$ of $\R^{V_Z}$ and the convex polytope $Y_Z$ in $Q_Z$ respectively by
\[
Q_Z:=\left\{\kappa \in \R^{V_Z} \ |\ \sum_{v\in V_Z} \kappa_v=0  \right\},\quad
Y_Z:=\bigcap_{\emptyset \neq U \subsetneq V_Z}\left\{\kappa \in Q_Z \ |\ \sum_{v\in U}\kappa_v >\phi_Z(U) \right\}.
\]
Note that $P_Z, Y_Z$ are homeomorphic to $\R^{|V_Z|-1}$.
Lemmas~\ref{likeDGB},~\ref{compati} yield $K_Z(P_Z)\subset Y_Z$, 
and the map $K_Z:P_Z \to Y_Z$ is continuously and injectively extended to the map between the one point compactifications of $P_Z$, $Y_Z$.
Then by the invariance of domain theorem, the extended map is an open map.
Since the one point compactifications of $P_Z$, $Y_Z$ are Hausdorff, compact and connected, 
the extended map is surjective hence $K_Z(P_Z)=Y_Z$.
By Lemma~\ref{less}, $Y_Z$ contains the origin.
This with Lemma \ref{injective} gives a unique $r\in P_Z$ such that $K_Z(r)=0$.
\end{proof}
%%%%%%%
%%%%%%%%%%%%%%%%%%%%%%%%%%%%%%%%%%%%%
\begin{remark}
Some lemmas in Subsection 3.1 are valid not only for $Z\in \mathcal{Z}$, but also $V$.
To see this, let $K_V$ be the curvature on the weighted cell-division $(T_0,\Theta_0)$ of $S$ defined in \eqref{T0}.
We moreover set 
\[
P_V:=\left\{r_v\in \R^{V_0}_{>0} \ |\ \prod_{v\in V_0} r_v=1 \right\}
\]
if $\chi(S)=0$, and $P_V:=\R^{V_0}_{>0}$ if $\chi(S)<0$.
Since $(T_0,\Theta_0)$ satisfies the conditions (II) and (IV), 
Theorem \ref{CLT} ensures a unique $r\in P_V$ such that $K_V(r)=0$ as well as Lemma~\ref{origin}.
\end{remark}

\begin{definition}
Let $r^\ast$ be a unique element in $\R^V_{>0}$ such that $r^\ast|_{V_Z}\in P_Z$ and $K_Z(r^\ast)=0$ for all $Z\in\mathcal{Z} \cup\{V\}$. 
\end{definition}
%%%%%%%%%%%%%%%%%%%%%%%%%%%%%%%%%%%%%
\begin{proposition}\label{cor}
Fix $v\in V_0$ and a smooth curve $r(t):[0,\infty) \to \R^V_{>0}$.
We have $\lim_{t\to \infty}K_u(r(t))=0$ for any $u\in V$ if and only if 
the curve $\rho(t):[0,\infty)\to \R^{V}_{>0}$ defined for $u\in V$ as 
\[
\rho_u(t):=
\begin{dcases}
\frac{r_u(t)}{r_v(t)}, & \text{if $\chi(S)=0$},\\
r_u(t), & \text{if $\chi(S)<0$}
\end{dcases}
\]
degenerates of order $r^\ast$.
\end{proposition}
%%%%%%%%%%%%%%%%%%%%%%%%%%%%%%%%%%%%%
%%%%%%%%%%%%%%%%%%%%%%%%%%%%%%%%%%%%%
\begin{proof}
Since the curvature is invariant under uniform scalings of the metric for $\chi(S)=0$, 
we have $K_u(r(t))=K_u(\rho(t))$ for any $u\in V$ and $t\geq 0$.

If $\{\rho(t)\}_{t\geq0}$ degenerates of order $r^\ast$,
then Corollary \ref{outin0} yields for any  $Z\in \mathcal{Z}\cup\{V\}$ and $u\in V_Z$ that  
\[
\lim_{t\to \infty} K_u(r(t))=\lim_{t\to \infty} K_u(\rho(t)) =K_{Z,u}(r^\ast)=0.
\]

Conversely, assume that $\lim_{t\to \infty}K_u(r(t))=0$ holds for any $u\in V$.
For $Z\in \mathcal{C}(V_N)$, we see that $Z\in \mathcal{Z}, Z'=\emptyset$ and
\[
\lim_{t\to \infty}\sum_{u\in Z=V_Z} K_u(\rho(t))=0=\phi(Z).
\]
The equality condition in Proposition \ref{topcurv} (see also Remark \ref{310})
together with $\phi(Z\setminus U)<0$ for any nonempty proper subset $U$ of $Z$ leads to 
\[
\lim_{t\to \infty} \frac{\rho_u(t)}{\rho_z(t)}\in (0,\infty), \quad
\lim_{t\to \infty} \frac{\rho_z(t)}{\rho_w(t)}=0
%\lim_{t\to \infty} \rho_z(t)=0
\]
for any $u,z\in V_Z,w \in V^Z$, 
and $\lim_{t\to \infty} \rho_u(t)=0$ if $\chi(S)<0$, 
in turn Corollary~\ref{outin} yields
\[
0=\lim_{t\to \infty} K_u(\rho(t))=K_{Z,u}\left(\lim_{t\to \infty} ({\rho_x(t)}/{\rho_z(t)})_{x\in V_Z}\right)
\]
for any $u,z\in V_Z$.
By Lemma~\ref{injective}, 
the limit of $\{({\rho_x(t)}/{\rho_z(t)})_{x\in V_Z}\}_{t\geq 0}$ at infinity is parallel to $(r^\ast_x)_{x\in V_Z}$. 
An inductive argument for $0\leq i \leq N-1$ completes the proof of the corollary.
\end{proof}

\subsection{Main Theorem}
We prove the following theorem, which implies Theorem \ref{main}.
%%%%%%%%%%%%%%%%%
\begin{theorem}\label{MainTheorem}
Let $(T,\Theta)$ be a weighted triangulation of $S$ 
such that $\phi(U)\leq0 $ for any $U\subset V$ and  $Z_T=\{z\in Z\ |\ Z\in \mathcal{Z}\} \neq \emptyset$.
Fix $v\in V_0$ and the combinatorial Ricci flow $\{r(t)\}_{t\geq0}$ with an initial data $r\in\R^V_{>0}$.
Then the curve $\rho(t):[0,\infty)\to \R^{V}_{>0}$ defined for $u\in V$ as 
\[
\rho_u(t):=
\begin{dcases}
\frac{r_u(t)}{r_v(t)}, & \text{if $\chi(S)=0$},\\
r_u(t), & \text{if $\chi(S)<0$}
\end{dcases}
\]
degenerates of order $r^\ast$.
\end{theorem}
%%%%%%%%%%%%%%
\begin{proof}
By Proposition \ref{cor}, it suffices to show $\lim_{t\to \infty}|K(r(t))| =0$. 
As mentioned in Subsection 2.3,  $\{X_S(r(t))\}_{t\geq0}$ is a gradient flow of $\psi$.
Let us define for $t\geq0$ 
\[
x^\ast(t):= X_S ((t^{-n(u)}r_u^\ast)_{u\in V}).
\]
Then the curve $\{R_S(x^\ast(t)) \}_{t\geq0}$ degenerates order of $r^\ast$ and 
$ \lim_{t \to \infty} |x^\ast(t)-X_S(r(0))|^2/t^2=0$.
We moreover deduce
\[
\lim_{t \to \infty} \nabla \psi(x^\ast(t))=\lim_{t \to \infty}K(R_S(x^\ast(t)))=0
\]
from Proposition~\ref{convex} and Corollary \ref{outin0}.
Since Proposition \ref{AGS} yields that
\[
|K(r(t))|^2=|\nabla \psi(X_S(r(t))|^2 
\leq |\nabla \psi(x^\ast(t))|^2 +\frac{1}{t^2} |x^\ast(t)-X_S(r(0))|^2,
\]
by letting $t\to \infty$, we obtain $\lim_{t\to \infty}|K(r(t))|^2 =0$.
%%%
%
\end{proof}

\begin{proof}[Proof of Theorem $\ref{main}$]
The first claim follows from Theorem \ref{MainTheorem}.
In the case of $\chi(S)=0$, $\prod_{v\in V}r_v(t)$ is constant in $t\geq0$, which implies that 
$\lim_{t \to \infty} r_v(t)=\infty$ for $v\in V_0$,
that is $\{r(t)\}_{t\geq0}$ does not converges on $\R^V_{\geq 0}$ at infinity.
To prove the rest of the claim, 
let $(T_0,\Theta_0)$ be the weighted cell-devision defined in \eqref{T0}.
Then a desired weighted triangulation is $(T_0,\Theta_0)$ itself if $\mathcal{C}(Z_T) \cap\mathcal{Z}_4 =\emptyset$,
and is obtained by adding diagonal edges of weight 0 to $(T_0, \Theta_0)$ if $\mathcal{C}(Z_T) \cap\mathcal{Z}_4 \neq \emptyset$.
\end{proof}

%%%%%%%%%%%%%%%%%%%%%%%%%%%%%%%%
%%%%%%%%%%%%%%%%%%%%%%%%%%%%%%%%%%%%%%
		

\begin{bibdiv}
  \begin{biblist}
%\scriptsize
\bib{AGS}{book}{
   author={Ambrosio, Luigi},
   author={Gigli, Nicola},
   author={Savar\'e, Giuseppe},
   title={Gradient flows in metric spaces and in the space of probability
   measures},
   series={Lectures in Mathematics ETH Z\"urich},
   edition={2},
   publisher={Birkh\"auser Verlag, Basel},
   date={2008},
   pages={x+334},
%   isbn={978-3-7643-8721-1},
%   review={\MR{2401600}},
}
%%%%%%%%%%%%%%%%%%%%%		
\bib{CL}{article}{
   author={Chow, Bennett},
   author={Luo, Feng},
   title={Combinatorial Ricci flows on surfaces},
   journal={J. Differential Geom.},
   volume={63},
   date={2003},
   number={1},
   pages={97--129},
%   issn={0022-040\xi},
%   review={\MR{2015261}},
}
%%%%%%%%%%%%%%%%%%%
\bib{V}{article}{
   author={Colin de Verdi\`ere, Yves},
   title={Un principe variationnel pour les empilements de cercles},
   %language={French},
   journal={Invent. Math.},
   volume={104},
   date={1991},
   number={3},
   pages={655--669},
%   issn={0020-9910},
%   review={\MR{1106755}},
%   doi={10.1007/BF01245096},
}
%%%%%%%%%%%%%%%%%%%%%%%
\bib{MR}{article}{
   author={Marden, Al},
   author={Rodin, Burt},
   title={On Thurston's formulation and proof of Andreev's theorem},
   conference={
      title={Computational methods and function theory},
      address={Valpara\'\i so},
      date={1989},
   },
   book={
      series={Lecture Notes in Math.},
      volume={1435},
      publisher={Springer, Berlin},
   },
   date={1990},
   pages={103--115},
%   review={\MR{1071766}},
%   doi={10.1007/BFb0087901},
}

\bib{Th}{book}{
   author={Thurston, William P.},
   title={Geometry and topology of three-manifolds},
   %series={http://www.msri.org/publications/books/gt3m/},
   %volume={35},
   %note={http://www.msri.org/publications/books/gt3m/},
   publisher={Princeton lecture notes,  available at {\sf http://www.msri.org/publications/books/gt3m/}},
   %date={2002},
   %pages={x+311},
%   isbn={http://www.msri.org/publications/books/gt3m/},
%   review={http://www.msri.org/publications/books/gt3m/},
}

\bib{Y}{article}{
   author={Yamada, Masahiro},
   title={Master Thesis},
   language={Japanese},
   journal={Department of Mathematics and Information Sciences, Tokyo Metropolitan University},
   date={2018},
   }

\end{biblist}
\end{bibdiv}
\end{document}